\documentclass[final]{siamltex}

\usepackage{epsfig}
\usepackage{amsmath,amssymb}
\usepackage {graphicx}
\usepackage{amsmath,amssymb}

\newtheorem{cor}{Corollary}[section]
\newtheorem{example}{Example}
\newtheorem{rem}{Remark}[section]

\title{On the correction equation of the Jacobi-Davidson method}

\author{Gang Wu\thanks{Department of Mathematics,
China University of Mining and Technology \& School of Mathematics and Statistics, Jiangsu Normal University, Xuzhou, 221116, Jiangsu, P.R. China.
E-mail: {\tt wugangzy@gmail.com} and {\tt gangwu76@126.com}. This author is
supported by the National Science Foundation of China under grant 11371176, the National Science Foundation of Jiangsu
Province under grant BK20131126, the 333 Project of Jiangsu Province, and the Talent Introduction Program of China
University of Mining and Technology.}
        \and Hong-kui Pang\thanks{Corresponding author. School of Mathematics and Statistics, Jiangsu Normal University, Xuzhou, 221116,
        Jiangsu, P.R. China. E-mail: {\tt panghongkui@163.com}. This author is supported by the National
Science Foundation of China under grant 11201192, and the
Natural Science Foundation of Jiangsu Province under grant
BK2012577.}
}

\begin{document}

\maketitle


\begin{abstract}
The Jacobi-Davidson method is one of the most popular approaches for iteratively computing a few eigenvalues and their associated eigenvectors of a large matrix.
The key of this method is to expand the search subspace via solving the Jacobi-Davidson correction equation, whose coefficient matrix is singular.
It is believed long by scholars that the Jacobi-Davidson correction equation is a consistent linear system. In this work, we point out that the correction equation may have a unique solution or have no solution at all, and we derive a computable necessary and sufficient condition for cheaply judging the existence and uniqueness of solution of the correction equation. Furthermore, we consider the difficulty of stagnation that bothers the Jacobi-Davidson method, and verify that if the Jacobi-Davidson method stagnates, then the corresponding Ritz value is a defective eigenvalue of the projection matrix. We provide a computable necessary and sufficient condition for expanding the search subspace successfully. The properties of the Jacobi-Davidson method with preconditioning and some alternative Jacobi-Davidson correction equations are also discussed.
\end{abstract}



\pagestyle{myheadings} \thispagestyle{plain} \markboth{G. WU AND H. PANG}{On correction equation of the Jacobi-Davidson method}

\section{Introduction}
\setcounter{equation}{0}
We are interested in computing a few eigenvalues and the corresponding eigenvectors of
an $n$-by-$n$ large matrix $A$. The Jacobi-Davidson method
is one of the most popular approaches for this type of problem, see \cite{Bai,Fro,GS,Huang,Jia,Mee,JD2,JD,S1,S2,Stewart,Van} and the references therein. In essence, this method can be understood as a Newton-based method \cite{Stewart}.
The Jacobi-Davidson method is based on two principles \cite{JD2,JD}: Given a subspace $V_k=[{\bf v}_1,{\bf v}_2,\ldots,{\bf v}_k]$ with $k\ll n$, the first principle is to apply
a Ritz-Galerkin approach on the large eigenproblem $A{\bf x}=\lambda {\bf x}$:
$$
AV_k{\bf y}_k-\widetilde{\lambda}_k V_k{\bf y}_k~\bot~{\rm span}\{V_k\},
$$
which reduces to a $k$-by-$k$ eigenproblem
$$
V_k^{\rm H}AV_k{\bf y}_k=\widetilde{\lambda}_k{\bf y}_k.
$$
Then this method makes use of $(\widetilde{\lambda}_k,{\bf u}_k=V_k{\bf y}_k)$ as an approximate eigenpair, called Ritz pair of $A$ in the subspace spanned by the columns of $V_k$.

The second principle is to modify the approximation from solving the Jacobi-Davidson correction equation for expanding ${\rm span}\{V_k\}$. More precisely,
for the approximate eigenvector ${\bf u}_k$, the Jacobi-Davidson method computes an {\it orthogonal} correction {\bf v} for ${\bf u}_k$, such that
$$
A({\bf u}_k+{\bf v})=\lambda({\bf u}_k+{\bf v}).
$$
As ${\bf v}\bot{\bf u}_k$, we focus on the subspace orthogonal to ${\bf u}_k$.
Let ${\bf r}_k=(A-\widetilde{\lambda}_k I){\bf u}_k$ be the residual, then ${\bf r}_k\bot{\bf u}_k$, and the orthogonal projection of $A$ onto the subspace ${\rm span}\{{\bf u_k}\}^{\bot}$ is $(I-{\bf u}_k{\bf u}_k^{\rm H})A(I-{\bf u}_k{\bf u}_k^{\rm H})$, where $(\cdot)^{\rm H}$ denotes the conjugate transpose of a matrix or vector. It is easy to check that the vector ${\bf v}$ satisfies
$$
(I-{\bf u}_k{\bf u}_k^{\rm H})(A-\lambda I)(I-{\bf u}_k{\bf u}_k^{\rm H}){\bf v}=-{\bf r}_k.
$$
Since the eigenvalue $\lambda$ is unknown, we replace it by $\widetilde{\lambda}_k$, which gives the following {\it Jacobi-Davidson correction equation} \cite{JD2,JD}:
\begin{equation}\label{1.1}
(I-{\bf u}_k{\bf u}_k^{\rm H})(A-\widetilde{\lambda}_k I)(I-{\bf u}_k{\bf u}_k^{\rm H}){\bf v}=-{\bf r}_k.
\end{equation}
Throughout this paper, we make the assumption that ${\bf v}\in {\rm span}\{{\bf u}_k\}^{\bot}$ and $\widetilde{\lambda}_k$ is not an eigenvalue of $A$.
If {\rm(\ref{1.1})} has a unique solution, then one solves the equation (\ref{1.1}) and expands the search subspace ${\rm span}\{V_k\}$ with ${\bf v}$, yielding a new subspace ${\rm span}\{V_{k+1}\}$. One then computes a new Ritz pair
$(\widetilde{\lambda}_{k+1},{\bf u}_{k+1})$ with respect to $V_{k+1}$, and repeats the above procedure. This is the basis of the Jacobi-Davidson method \cite{JD2,JD};
for more details and implementations on this method, refer to \cite{Bai,JD2,JD,Stewart,Van}.

In \cite{GS}, Genseberger and Sleijpen proposed an {\it alternative} correction equation for the Jacobi-Davidson method. Consider some subspace $\mathcal{W}$ such that ${\bf u}_k\in\mathcal{W}\subseteq {\rm span}\{V_k\}$. Let $W$ be an orthonormal basis for $\mathcal{W}$, the alternative correction equation reads \cite{GS}
\begin{equation}\label{eqn23}
(I-WW^{\rm H})(A-\widetilde{\lambda}_k I)(I-WW^{\rm H}){\bf v}=-{\bf r}_k.
\end{equation}
In this situation, we have ${\bf v}\in {\rm span}\{W\}^{\bot}$ and ${\bf r}_k\perp {\rm span}\{V_k\}$, where ${\rm span}\{W\}^{\bot}$ represents the orthogonal complement of ${\rm span}\{W\}$. Obviously, if $\mathcal{W}={\rm span}\{{\bf u}_k\}$, then (\ref{eqn23}) reduces to the Jacobi-Davidson correction equation (\ref{1.1}). On the other hand,
if we choose $\mathcal{W}={\rm span}\{V_k\}$, then (\ref{eqn23}) reads
\begin{equation}\label{eqn24}
(I-V_kV_k^{\rm H})(A-\widetilde{\lambda}_k I)(I-V_kV_k^{\rm H}){\bf v}=-{\bf r}_k.
\end{equation}

Another alternative is the two-sided Jacobi-Davidson method inspired by the two-sided Rayleigh quotient iteration \cite{TS}. In the two-sided Jacobi-Davidson method, we work with two
search spaces, $\mathcal{Q}$ for the right subspace and $\mathcal{P}$ for the left subspace. Then the two-sided Jacobi-Davidson method applies the Petrov-Galerkin conditions \cite{Stewart} on
the right residual ${\bf r}_{q}$ and left residual ${\bf r}_p$ to determine approximate eigenvectors ${\bf q}$ and ${\bf p}$, where
$$
{\bf r}_q=A{\bf q}-\theta{\bf q}~\bot~\mathcal{P},\quad {\bf r}_p=A^{\rm H}{\bf p}-\bar{\theta}{\bf p}~\bot~\mathcal{Q},
$$
${\bf q}\in\mathcal{Q}$ and ${\bf p}\in\mathcal{P}$ are approximations to the
right and left eigenvectors, ${\bf q}^{\rm H}{\bf p}\neq 0$, and $\theta=\frac{{\bf q}^{\rm H}A{\bf p}}{{\bf q}^{\rm H}{\bf p}}$ is an approximation to the eigenvalue with $\bar{\theta}$ being the conjugate of $\theta$. Let $Q$ and $P$ form bases for $\mathcal{Q}$ and $\mathcal{P}$, respectively, Hochstenbach and Sleijpen considered two variants of the Jacobi-Davidson method. In the first variant, we want the columns of $Q$ and $P$ to be {\it bi-orthogonal}, i.e., $Q^{\rm H}P$ is a diagonal matrix. Then the {\it right} correction equation reads
\begin{equation}\label{114}
\left(I-\frac{{\bf q}\cdot{\bf p}^{\rm H}}{{\bf p}^{\rm H}{\bf q}}\right)(A-\theta I)\left(I-\frac{{\bf q}\cdot{\bf p}^{\rm H}}{{\bf p}^{\rm H}{\bf q}}\right){\bf s}=-{\bf r}_q,\quad {\bf s}\bot{\bf p},
\end{equation}
and the {\it left} correction equation reads
\begin{equation}\label{115}
\left(I-\frac{{\bf p}\cdot{\bf q}^{\rm H}}{{\bf q}^{\rm H}{\bf p}}\right)(A^{\rm H}-\bar{\theta} I)\left(I-\frac{{\bf p}\cdot{\bf q}^{\rm H}}{{\bf q}^{\rm H}{\bf p}}\right){\bf t}=-{\bf r}_p,\quad {\bf t}\bot{\bf q}.
\end{equation}
Without lose of generality, we assume that ${\bf q}^{\rm H}{\bf p}=1$ and $Q,P$ are {\it bi-orthonormal}, i.e, $Q^{\rm H}P=P^{\rm H}Q=I$, then (\ref{114}) and (\ref{115}) can be rewritten as
\begin{equation}\label{1.4}
(I-{\bf q}{\bf p}^{\rm H})(A-\theta I)(I-{\bf q}{\bf p}^{\rm H}){\bf s}=-{\bf r}_q,\quad {\bf s}\bot{\bf p},
\end{equation}
and
\begin{equation}\label{1.5}
(I-{\bf p}{\bf q}^{\rm H})(A^{\rm H}-\bar{\theta} I)(I-{\bf p}{\bf q}^{\rm H}){\bf t}=-{\bf r}_p,\quad {\bf t}\bot{\bf q},
\end{equation}
respectively.

In the second variant, we keep the columns of both $Q$ and $P$ {\it orthogonal}, and the the two correction equations now take the form
\begin{equation}\label{1.6}
\left(I-\frac{{\bf q}\cdot{\bf p}^{\rm H}}{{\bf p}^{\rm H}{\bf q}}\right)(A-\theta I)(I-{\bf q}{\bf q}^{\rm H}){\bf s}=-{\bf r}_q,\quad {\bf s}\bot{\bf q},
\end{equation}
and
\begin{equation}\label{1.7}
\left(I-\frac{{\bf p}\cdot{\bf q}^{\rm H}}{{\bf q}^{\rm H}{\bf p}}\right)(A^{\rm H}-\bar{\theta} I)(I-{\bf p}{\bf p}^{\rm H}){\bf t}=-{\bf r}_p,\quad {\bf t}\bot{\bf p}.
\end{equation}

Although there is a large amount of work has been performed on the Jacobi-Davidson method since the 1990's, there is still some problems needs to be resolved. For instance, it is believed that (\ref{1.1}) is a consistent linear equation (\cite[pp.89]{Bai},\cite[pp.399]{Stewart}), and ${\bf v}$ has a unique solution.
In this paper, we show that the Jacobi-Davidson correction equation (\ref{1.1}) may have a unique solution or have no solution. Thus, its necessary to investigate the existence and uniqueness of the solutions of the correction equations (\ref{1.1})--(\ref{1.7}).
Another problem that bothers the Jacobi-Davidson method is stagnation \cite{GS}, i.e., one may fail to expand the search subspace even if the correction equation has a unique solution. To our best knowledge, however, there is few theoretical results on this topic, and more theoretical analysis is required.

This paper is organized as follows. In Section 2, we investigate the existence and uniqueness of the solutions of the correction equations (\ref{1.1})--(\ref{1.7}). We conclude that the Jacobi-Davidson correction equation (\ref{1.1}) and the two-sided Jacobi-Davidson correction equations (\ref{1.4})--(\ref{1.7}) may have a unique solution or have no solution. We
show that the alternative correction equation (\ref{eqn23}) may have a unique solution, infinite number of solutions or no solution. Furthermore, we provide {computable} necessary and sufficient conditions for cheaply judging the existence and uniqueness of the solutions of the correction equations. In Section 3, we take into account the problem of stagnation, and indicate that the Jacabi-Davidson method based on (\ref{1.1})--(\ref{1.7}) other than (\ref{eqn24}) may suffer from stagnation. It is shown that the Ritz value $\widetilde{\lambda}_k$ being defective is a sufficient condition for the stagnation of the Jacobi-Davidson method.
Some computable necessary and sufficient conditions for expanding the search subspace successfully are proposed.

\section{Existence and uniqueness of the solutions}
\setcounter{equation}{0}


As is well known that the Jacobi-Davidson correction equation is the key ingredient of the Jacobi-Davidson method \cite{JD2,JD},
however, the following example shows that it can be inconsistent.
\begin{example}
Consider the $3\times 3$ matrix
$$
A=\left[\begin{array}{ccc}
1 & 1 & 1\\ 1 & 1 & 1\\ 0 & 0 & 2
\\\end{array} \right].
$$
If we choose $V_1=\big[1,0,0\big]^{\rm H}={\bf u}_1$, then $\widetilde{\lambda}_1={\bf u}_1^{\rm H}A{\bf u}_1=1$,
${\bf r}_1=A{\bf u}_1-{\bf u}_1=\big[0,1,0\big]^{\rm H}~\bot~{\bf u}_1$, and
$$
A-\widetilde{\lambda}_1I=
\left[\begin{array}{ccc}
0 & 1 & 1\\
1 & 0 & 1\\
0 & 0 & 1\\
\end{array} \right],
\quad I-{\bf u}_1{\bf u}_1^{\rm H}=\left[\begin{array}{ccc}
0 & 0 & 0\\
0 & 1 & 0\\
0 & 0 & 1\\
\end{array} \right].
$$
Denote ${\bf v}=[\bar{v}_1,\bar{v}_2,\bar{v}_3]^{\rm H}$, then the Jacobi-Davidson correction equation {\rm(\ref{1.1})} turns out to be
$$
\left[\begin{array}{ccc}
0 & 0 & 0\\
0 & 0 & 1\\
0 & 0 & 1\\
\end{array} \right]\left[\begin{array}{ccc}
v_1\\
v_2\\
v_3\\
\end{array} \right]=\left[\begin{array}{ccc}
0\\
-1\\
0\\
\end{array} \right],
$$
which has no solution. Thus, the Jacobi-Davidson correction equation is inconsistent, and we will fail to expand the search subspace by solving the Jacobi-Davidson correction equation.
\end{example}

We remark that the inconsistent problem of (\ref{1.1}) has also been mentioned by Feng and Jia \cite{Feng}. They provided the following properties on the solution of correction
equation.
\begin{theorem}\cite[Theorem 1]{Feng}\label{Thm11}
Assume that $(\rho,{\bf u})$ is an approximate eigenpair of the matrix $A$ with ${\bf u}\in {\rm span}\{V_k\}$ and $\rho={\bf u}^{\rm H}A{\bf u}$, and select a matrix $U_{\bot}$ such that $[{\bf u},~U_{\bot}]$ is unitary. Then the columns of $U_{\bot}$ form
an orthonormal basis of ${\rm span}\{\bf u\}^{\bot}$. Set ${\bf r} = (A - \rho I){\bf u}$. Then ${\bf r}\bot{\bf u}$, and there exists a unique ${\bf b}$
such that ${\bf r} = U_{\bot}{\bf b}$. For the linear system
\begin{equation}\label{eq2.1}
(I-{\bf u}{\bf u}^{\rm H})(A-\rho I)(I-{\bf u}{\bf u}^{\rm H}){\bf t}=-{\bf r},\quad {\bf t}\bot{\bf u}.
\end{equation}
there hold the following results:\\
{\rm(i)} equation (\ref{eq2.1}) has no solution if ${\bf b}\notin \mathcal{R}(U_{\bot}^{\rm H}AU_{\bot}-\rho I)$, where $\mathcal{R}(\cdot)$ denotes the range of a matrix;\\
{\rm(ii)} equation (\ref{eq2.1}) has at least one solution if ${\bf b}\in \mathcal{R}(U_{\bot}^{\rm H}AU_{\bot}-\rho I)$;\\
{\rm(iii)} equation (\ref{eq2.1}) has a unique solution if and only if $\rho$ is not an eigenvalue of $U_{\bot}^{\rm H}AU_{\bot}$.
\end{theorem}

Theorem \ref{Thm11} indicates that the correction equation {\rm(}\ref{1.1}{\rm)} may have a unique solution, infinite number of solutions or no solution. We revisit this problem and point out that the Jacobi-Davidson correction equation {\rm(}\ref{1.1}{\rm)} either has a unique solution or has no solution, and it will never has infinite number of solutions theoretically.
Furthermore, the conditions provided by Theorem \ref{Thm11} for judging the existence of solution of the correction equation are difficult to use in practice.
In this section, we aim to cheaply judge the existence of solution of the correction equations, and provide computable necessary and sufficient conditions for expanding the search subspace successfully.

As the Jacobi-Davidson correction equation (\ref{1.1}) is a special case of the alternative correct equation, it is sufficient to consider \eqref{eqn23}. We have the following thorem.
\begin{theorem}\label{Thm12}
Suppose that ${\rm span}\{{\bf u}_k\}\subseteq\mathcal{W}\subseteq {\rm span}\{V_k\}$, let $W$ be an orthonormal basis for $\mathcal{W}$, and denote ${\bf z}_k=W^{\rm H}{\bf u}_k$. Then\\
{\rm(i)} the alternative correction equation {\rm(\ref{eqn23})} has a unique solution if and only if $W^{\rm H}(A-\widetilde{\lambda}_kI)^{-1}W$ is nonsingular.\\
{\rm(ii)} the alternative correction equation {\rm(\ref{eqn23})} has no solution if and only if $W^{\rm H}(A-\widetilde{\lambda}_kI)^{-1}W$ is singular and ${\bf z}_k\notin {\rm span}\{W^{\rm H}(A-\widetilde{\lambda}_kI)^{-1}W\}$.\\
{\rm(iii)} the alternative correction equation {\rm(\ref{eqn23})} has infinite number of solutions if and only if $W^{\rm H}(A-\widetilde{\lambda}_kI)^{-1}W$ is singular and ${\bf z}_k\in{\rm span}\{W^{\rm H}(A-\widetilde{\lambda}_kI)^{-1}W\}$.\\
\end{theorem}
\begin{proof}
Let $\big[W,~\widehat{U}_{\bot}\big]$ be unitary,
then (\ref{eqn23}) can be reformulated as
$$
\widehat{U}_{\bot}\big[\widehat{U}_{\bot}^{\rm H}(A-\widetilde{\lambda}_k I)\widehat{U}_{\bot}\big]\widehat{U}_{\bot}^{\rm H}{\bf v}=-\widehat{U}_{\bot}\widehat{U}_{\bot}^{\rm H}{\bf r}_k,
$$
where we used the fact ${\bf r}_k=\widehat{U}_{\bot}\widehat{U}_{\bot}^{\rm H}{\bf r}_k$. As $\widehat{U}_{\bot}\in\mathbb{C}^{n\times (n-1)}$ is orthonormal, we can rewrite the above equation as
$$
\big[\widehat{U}_{\bot}^{\rm H}(A-\widetilde{\lambda}_k I)\widehat{U}_{\bot}\big]\widehat{U}_{\bot}^{\rm H}{\bf v}=-\widehat{U}_{\bot}^{\rm H}{\bf r}_k.
$$
Denote ${\bf \widehat{w}}=\widehat{U}_{\bot}^{\rm H}{\bf v}$, then we have
\begin{equation}\label{Uwrk}
\widehat{U}_{\bot}^{\rm H}\big[(A-\widetilde{\lambda}_k I)\widehat{U}_{\bot}{\bf \widehat{w}}+{\bf r}_k\big]=0.
\end{equation}
Since
$$
{\bf r}_k = (A-\widetilde{\lambda}_k I){\bf u}_k = (A-\widetilde{\lambda}_k I)W(W^{\rm H}{\bf u}_k) = (A-\widetilde{\lambda}_k I)W{\bf z}_k,
$$
it follows from \eqref{Uwrk} that
$$
\widehat{U}_{\bot}^{\rm H}(A-\widetilde{\lambda}_k I)\big[W,~\widehat{U}_{\bot}\big]\left[
                                                                                              \begin{array}{c}
                                                                                               {\bf z}_k \\
                                                                                               {\bf \widehat{w}} \\
                                                                                                  \end{array}
                                                                                                \right]
=0.
$$
Or equivalently, there exists a vector ${\bf \widehat{z}}$, such that
$$
(A-\widetilde{\lambda}_k I)\big[W,~\widehat{U}_{\bot}\big]\left[
                                                                 \begin{array}{c}
                                                                  {\bf z}_k \\
                                                                  {\bf \widehat{w}} \\
                                                                 \end{array}
                                                          \right]=W{\bf \widehat{z}},
$$
from which we obtain
\begin{align}\label{detereq}
\left[
     \begin{array}{c}
     {\bf z}_k \\
     {\bf \widehat{w}} \\
     \end{array}
\right]&=\left[
          \begin{array}{c}
           W^{\rm H}\\
           \widehat{U}_{\bot}^{\rm H}\\
          \end{array}
        \right](A-\widetilde{\lambda}_k I)^{-1}W{\bf \widehat{z}}\nonumber\\
        &=\left[
           \begin{array}{c}
           W^{\rm H}(A-\widetilde{\lambda}_k I)^{-1}W{\bf \widehat{z}} \\
           \widehat{U}_{\bot}^{\rm H}(A-\widetilde{\lambda}_k I)^{-1}W{\bf \widehat{z}} \\
           \end{array}
           \right].
\end{align}
Recall that ${\bf v}\in {\rm span}\{\widehat{U}_{\bot}\}$ and ${\bf \widehat{w}}=\widehat{U}_{\bot}^{\rm H}{\bf v}$, which implies that ${\bf v}$ is uniquely determined by ${\bf \widehat{w}}$. Thus, in terms of \eqref{detereq} we get the results of (i), (ii), and (iii).%
\end{proof}
\begin{rem}
Theorem \ref{Thm12} indicates that the existence and uniqueness of the solution to {\rm(}\ref{eqn23}{\rm)} and {\rm(}\ref{eqn24}{\rm)} cannot be guaranteed theoretically. In other words, the alternative Jacobi-Davidson correction equations may suffer from the difficulty of having no solution or having infinte number of solutions.
\end{rem}

As an extremal case of the alternative correction equation, the Jacobi-Davidson correction equation {\rm(\ref{1.1})} is achieved when $\mathcal{W}={\rm span}\{{\bf u}_k\}$. In this case, $W={\bf u}_k$, ${\bf z}_k={\bf u}_k^{\rm H}{\bf u}_k=1$, and \eqref{detereq} reduces to
\begin{align*}
\left[
     \begin{array}{c}
     1 \\
     {\bf \widehat{w}} \\
     \end{array}
\right]&=\left[
          \begin{array}{c}
           {\bf u}_k^{\rm H}\\
           \widehat{U}_{\bot}^{\rm H}\\
          \end{array}
        \right](A-\widetilde{\lambda}_k I)^{-1}{\bf u}_k\cdot{\bf \beta}\\
        &=\left[
           \begin{array}{c}
           {\bf u}_k^{\rm H}(A-\widetilde{\lambda}_k I)^{-1}{\bf u}_k\cdot{\bf \beta} \\
           \widehat{U}_{\bot}^{\rm H}(A-\widetilde{\lambda}_k I)^{-1}{\bf u}_k\cdot{\bf \beta} \\
           \end{array}
           \right],
\end{align*}
where $\beta\in\mathbb{C}$ is a scalar. In light of the above equation, we draw the follow conclusion.
\begin{cor}\label{Thm1}
The Jacobi-Davidson correction equation {\rm(\ref{1.1})} either has a unique solution or has no solution. More precisely,\\
{\rm(i)}~The Jacobi-Davidson correction equation {\rm(\ref{1.1})} has a unique solution if and only if ${\bf u}_k^{\rm H}(A-\widetilde{\lambda}_kI)^{-1}{\bf u}_k\neq 0$.\\
{\rm(ii)}~The Jacobi-Davidson correction equation {\rm(\ref{1.1})} has no solution if and only if ${\bf u}_k^{\rm H}(A-\widetilde{\lambda}_kI)^{-1}{\bf u}_k= 0$.
\end{cor}

\begin{rem}
Some remarks on Theorem \ref{Thm1} are also in order.
First, it is known that the correction equation is almost always solved approximately in practice \cite{JD,Stewart,Van}, however, if the linear system is inconsistent, then the obtained approximations may fail to expand the search subspace efficiently. This explains in some degree why the Jacobi-Davidson method converges irregularly in many cases \cite{Feng}.
Second, an interesting question is, if no solution exists, how to effectively expand the projection subspace? As ${\bf r}_k\bot{\bf u}_k$, one can exploit the residual ${\bf r}_k$ to expand the search subspace, and the Jacobi-Davidson method reduces to a residual expansion method \cite{Wu,Ye} in this case. Third, recall that the condition ${\bf u}_k^{\rm H}(A-\widetilde{\lambda}_kI)^{-1}{\bf u}_k\neq 0$ is
needed in the derivation of the Jacobi Davidson method \cite{JD,Stewart,Van}. Theorem \ref{Thm1} shows that this condition is indeed a necessary and sufficient condition for the Jacobi-Davidson correction equation to have a unique solution.
\end{rem}

Next, we take into account the two-sided and alternating Jacobi-Davidson correction equations (\ref{1.4})--(\ref{1.7}). The following theorem indicates that these correction equations may be inconsistent. The proof is similar to that of Corollary \ref{Thm1}, and thus is omitted.
\begin{theorem}
Under the notations in Section I, for the two-sided and alternating Jacobi-Davidson correction equations, we have that\\
{\rm(i)}~The two-sided Jacobi-Davidson correction equations {\rm(}\ref{1.4}{\rm)} and {\rm(}\ref{1.5}{\rm)} either have unique solutions or have no solutions. They have unique solutions if and only if ${\bf p}^{\rm H}(A-\theta I)^{-1}{\bf q}\neq 0$, and have no solutions if and only if ${\bf p}^{\rm H}(A-\theta I)^{-1}{\bf q}=0$.\\
{\rm(ii)}~The two-sided Jacobi-Davidson right correction equation {\rm(}\ref{1.6}{\rm)} either has a unique solution or has no solution. It has a unique solution if and only if ${\bf q}^{\rm H}(A-\theta I)^{-1}{\bf q}\neq 0$, and has no solution if and only if ${\bf q}^{\rm H}(A-\theta I)^{-1}{\bf q}=0$.\\
{\rm(iii)}~The two-sided Jacobi-Davidson left correction equation {\rm(}\ref{1.7}{\rm)} either has a unique solution or has no solution. It has a unique solution if and only if ${\bf p}^{\rm H}(A-\theta I)^{-1}{\bf p}\neq 0$, and has no solution if and only if ${\bf p}^{\rm H}(A-\theta I)^{-1}{\bf p}=0$.
\end{theorem}

\section{On stagnation of the Jacobi-Davidson methods}
\setcounter{equation}{0}

In this section, we focus on the difficulty of stagnation that bothers the Jacobi-Davidson method.
Without loss of generality, we make the assumption that ${\bf u}_k$ is the first column of $V_k$. Otherwise, let ${\bf e}_1$ be the first column of the identity matrix, and $H_k$ be a Householder matrix such that $H_k^{\rm H}{\bf y}_k={\bf e}_1$, then ${\bf u}_{k}=V_k{\bf y}_k=(V_kH_k){\bf e}_1$, and we make use of $\widehat{V}_k\equiv V_kH_k$ as the basis of the current search subspace. Notice that ${\rm span}\{\widehat{V}_k\}={\rm span}\{{V}_k\}$.
Denote $V_k=[{\bf u}_k,~\widehat{V}_2]$, where $\widehat{V}_2\in\mathbb{C}^{n\times (k-1)}$, and let $[{\bf u}_k,~U_{\bot}]$ be unitary, where $U_{\bot}=\big[\widehat{V}_2,~\widehat{V}_3\big]$ with $\widehat{V}_3\in\mathbb{C}^{n\times(n-k)}$.
Recall that
\begin{equation}\label{eqn2.1}
\big[U_{\bot}^{\rm H}(A-\widetilde{\lambda}_k I)U_{\bot}\big]U_{\bot}^{\rm H}{\bf v}=-U_{\bot}^{\rm H}{\bf r}_k,
\end{equation}
has a unique solution if and only if $U_{\bot}^{\rm H}(A-\widetilde{\lambda}_k I)U_{\bot}$ is nonsingular. Decompose
$$
\big[U_{\bot}^{\rm H}(A-\widetilde{\lambda}_k I){U}_{\bot}\big]^{-1}\equiv\left[\begin{array}{cc}
B_{11} & B_{12}\\
B_{21} & B_{22}\\
\end{array} \right],
$$
where $B_{11}\in\mathbb{C}^{(k-1)\times(k-1)}$, $B_{12}\in\mathbb{C}^{(k-1)\times(n-k)}$, $B_{21}\in\mathbb{C}^{(n-k)\times(k-1)}$, and $B_{22}\in\mathbb{C}^{(n-k)\times(n-k)}$. As ${\bf r}_k\bot V_k$, we obtain from \eqref{eqn2.1} that
\begin{eqnarray}\label{eqn2.3}
{\bf v}={U}_{\bot}{U}_{\bot}^{\rm H}{\bf v}&=&-{U}_{\bot}\big[{U}_{\bot}^{\rm H}(A-\widetilde{\lambda}_k I){U}_{\bot}\big]^{-1}{U}_{\bot}^{\rm H}{\bf r}_k\nonumber\\
&=&-\big[\widehat{V}_2,\widehat{V}_3\big]\left[\begin{array}{cc}
B_{11} & B_{12}\\
B_{21} & B_{22}\\
\end{array} \right]\left[\begin{array}{c}
{\bf 0}\\ \widehat{V}_3^{\rm H}{\bf r}_k
\end{array} \right]\nonumber\\
&=&-\widehat{V}_2B_{12}\widehat{V}_3^{\rm H}{\bf r}_k-\widehat{V}_3B_{22}\widehat{V}_3^{\rm H}{\bf r}_k.
\end{eqnarray}
As a consequence, we obtain
\begin{eqnarray}\label{eqn2.4}
V_kV_k^{\rm H}{\bf v}&=&-\big[{\bf u}_k,~\widehat{V}_2\big]\left[\begin{array}{c}
{\bf u}_k^{\rm H}\\ \widehat{V}_2^{\rm H}
\end{array} \right]\big[\widehat{V}_2,~\widehat{V}_3\big]\left[\begin{array}{c}
B_{12}\\B_{22}
\end{array} \right]\widehat{V}_3^{\rm H}{\bf r}_k\nonumber\\
&=&-\big[{\bf u}_k,~\widehat{V}_2\big]\left[\begin{array}{c}
{\bf 0}\\B_{12}
\end{array} \right]\widehat{V}_3^{\rm H}{\bf r}_k\nonumber\\
&=&-\widehat{V}_2B_{12}\widehat{V}_3^{\rm H}{\bf r}_k.
\end{eqnarray}
Note that ${\bf v}\in {\rm span}\{V_k\}$ if and only if $V_kV_k^{\rm H}{\bf v}={\bf v}$. In view of (\ref{eqn2.3}) and (\ref{eqn2.4}), we get the following result.
\begin{theorem}\label{Thm2}
Under the above notations, if {\rm (\ref{1.1})} has a unique solution, then
\begin{equation}
{\bf v}\notin {\rm span}\{V_k\}\Longleftrightarrow {\bf r}_k\notin \mathcal{N}(\widehat{V}_3B_{22}\widehat{V}_3^{\rm H}),
\end{equation}
where $\mathcal{N}(\widehat{V}_3B_{22}\widehat{V}_3^{\rm H})$ denotes the null space of $\widehat{V}_3B_{22}\widehat{V}_3^{\rm H}$.
\end{theorem}

\begin{example}
Consider the $4\times 4$ matrix
$$
A=\left[\begin{array}{cccc}
1 & 1 & 2 & 3\\ 0& 1 & 2 & -1\\ 0& 0& -2 & 2 \\ 1 &1/3 &4/3 & 0\\
\end{array} \right].
$$
If we choose
$$
V_2=\left[\begin{array}{cc}
1 & 0\\ 0& 1\\ 0& 0\\ 0 &0
\end{array} \right],
$$
then $\widetilde{\lambda}_2=1$ is a Ritz value and ${\bf u}_2=\big[1,0,0,0\big]^{\rm H}$ is a Ritz vector of $A$ in the search subspace ${\rm span}\{V_2\}$, which is a defective eigenvalue of $V_2^{\rm H}AV_2$. Thus,
${\bf r}_2=A{\bf u}_2-{\bf u}_2=\big[0,0,0,1\big]^{\rm H}~\bot~{\rm span}\{V_2\}$, and
$$
U_{\bot}=\left[\begin{array}{ccc}
0 & 0 & 0\\
1 & 0 & 0\\
0 & 1 & 0\\
0 & 0 & 1\\
\end{array} \right],
\quad
U_{\bot}^{\rm H}(A-\widetilde{\lambda}_2I)U_{\bot}=
\left[\begin{array}{ccc}
0 & 2 & -1\\
0 & -3 & 2\\
1/3 & 4/3 & -1\\
\end{array} \right].
$$
Therefore, \eqref{eqn2.1} has a unique solution
$U_{\bot}^{\rm H}{\bf v}=\big[-3,0,0\big]^{\rm H}$, and the expansion vector
$$
{\bf v}=U_{\bot}U_{\bot}^{\rm H}{\bf v}=\big[0,-3,0,0\big]^{\rm H}\in{\rm span}\{V_2\}.
$$
Consequently, we fail to expand the search subspace ${\rm span}\{V_2\}$ by using the solution of the Jacobi-Davidson correction equation.
Indeed, we note that
$$
\widehat{V}_3B_{22}\widehat{V}_3^{\rm H}=\left[\begin{array}{cccc}
0 & 0 & 0 & 0\\ 0& 0 & 0 & 0\\ 0& 0& 1 & 0 \\ 0 &0 &2 & 0\\
\end{array} \right],
$$
$\widehat{V}_3B_{22}\widehat{V}_3^{\rm H}{\bf r}_2={\bf 0}$, and it follows from Theorem \ref{Thm2} that ${\bf v}\in {\rm span}\{V_2\}$.
\end{example}

Unfortunately, Theorem \ref{Thm2} is difficult to use in practice. In this section, we aim to derive a computable necessary and sufficient condition for ${\bf v}\notin {\rm span}\{V_k\}$, i.e., $V_kV_k^{\rm H}{\bf v}\neq {\bf v}$, such that one can expand the search subspace successfully. We have the following theorem.
\begin{theorem}\label{Thm2.5}
Suppose that the Jacobi-Davidson correction equation {\rm(\ref{1.1})} has a unique solution, then
\begin{equation}
{\bf v}\notin {\rm span}\{V_k\}\Longleftrightarrow {\bf u}_k\notin {\rm span}\{(A-\widetilde{\lambda}_kI)V_k\}.
\end{equation}
\end{theorem}
\begin{proof}
If {\rm(\ref{1.1})} has a unique solution, then \cite{JD2,JD}
$$
{\bf v}=-(A-\widetilde{\lambda}_kI)^{-1}{\bf r}_k+\alpha(A-\widetilde{\lambda}_kI)^{-1}{\bf u}_k,
$$
where $\alpha=\frac{{\bf u}_k^{\rm H}(A-\widetilde{\lambda}_kI)^{-1}{\bf u}_k}{{\bf u}_k^{\rm H}(A-\widetilde{\lambda}_kI)^{-1}{\bf u}_k}$. Note that $\alpha\neq 0$. Otherwise, we will have ${\bf v}=-{\bf u}_k$, which is a contradiction. Thus,
\begin{eqnarray*}
V_kV_k^{\rm H}{\bf v}&=&-V_kV_k^{\rm H}{\bf u}_k+\alpha V_kV_k^{\rm H}(A-\widetilde{\lambda}_kI)^{-1}{\bf u}_k\\
&=&-{\bf u}_k+\alpha V_kV_k^{\rm H}(A-\widetilde{\lambda}_kI)^{-1}{\bf u}_k,
\end{eqnarray*}
and $V_kV_k^{\rm H}{\bf v}={\bf v}$ if and only if
$$
\alpha V_kV_k^{\rm H}(A-\widetilde{\lambda}_kI)^{-1}{\bf u}_k=\alpha (A-\widetilde{\lambda}_kI)^{-1}{\bf u}_k.
$$
As $\alpha\neq 0$, we arrive at
$$
V_kV_k^{\rm H}(A-\widetilde{\lambda}_kI)^{-1}{\bf u}_k=(A-\widetilde{\lambda}_kI)^{-1}{\bf u}_k.
$$
As a result, ${\bf v}\notin {\rm span}\{V_k\}$ {\it iff} $(A-\widetilde{\lambda}_kI)^{-1}{\bf u}_k\notin {\rm span}\{V_k\}$, or ${\bf u}_k\notin {\rm span}\{(A-\widetilde{\lambda}_kI)V_k\}$.
\end{proof}

It was mentioned in \cite{GS} that if the Ritz value $\widetilde{\lambda}_k$ is defective, then the correction equation {\rm(\ref{1.1})} may have a solution
in the current search subspace \cite[pp.237]{GS}. In such a case the search subspace is not expanded and Jacobi-Davidson
stagnates. An interesting question is whether the defectiveness of $\widetilde{\lambda}_k$ is a necessary and sufficient condition for stagnation of the Jacobi-Davidson method.
The following theorem indicates that $\widetilde{\lambda}_k$ being defective is a sufficient condition to the stagnation of the Jacobi-Davidson method.
\begin{theorem}\label{Cor3.1}
If the Jacobi-Davidson method stagnates, then $\widetilde{\lambda}_k$ is a defective eigenvalue of $V_k^{\rm H}AV_k$.
\end{theorem}
\begin{proof}
If the Jacobi-Davidson method stagnates, i.e., ${\bf v}\in{\rm span}\{V_k\}$, then we have from Theorem \ref{Thm2.5} that ${\bf u}_k\in{\rm span}\{(A-\widetilde{\lambda}_kI)V_k\}$. Thus, there is a non-zero vector ${\bf h}_k$, such that ${\bf u}_k=V_k{\bf y}_k=(A-\widetilde{\lambda}_kI)V_k{\bf h}_k$, and
$$
{\bf y}_k=V_k^{\rm H}(A-\widetilde{\lambda}_kI)V_k{\bf h}_k=(T_k-\widetilde{\lambda}_kI){\bf h}_k,
$$
where $T_k=V_k^{\rm H}AV_k$ is the projection matrix. Since $(\widetilde{\lambda}_k,{\bf y}_k)$ is an eigenpair of $T_k$, we obtain
\begin{equation}\label{3.6}
(T_k-\widetilde{\lambda}_kI){\bf y}_k=(T_k-\widetilde{\lambda}_kI)^2{\bf h}_k={\bf 0}.
\end{equation}
That is,
\begin{equation}\label{3.7}
(T_k-\widetilde{\lambda}_kI){\bf h}_k\neq{\bf 0}\quad {\rm and}\quad (T_k-\widetilde{\lambda}_kI)^2{\bf h}_k={\bf 0},\quad {\bf h}_k\neq{\bf 0}.
\end{equation}
As a result, $\widetilde{\lambda}_k$ is a defective eigenvalue of $T_k$. Indeed, let $T_k=X^{-1}JX$ be the Jordan decomposition of $T_k$, and the Jordan canonical form
$$
J=\left[\begin{array}{cc}
J_1 & \\
 & J_2\\
\end{array} \right]\in\mathbb{C}^{k\times k},
$$
where $J_1\in\mathbb{C}^{d_k\times d_k}$ is the Jordan block corresponding to $\widetilde{\lambda}_k$, with $d_k$ being the multiplicity of $\widetilde{\lambda}_k$; and $J_2\in\mathbb{C}^{(k-d_k)\times (k-d_k)}$ consists of the Jordan blocks corresponding to eigenvalues other than $\widetilde{\lambda}_k$. So we have from \eqref{3.6} that
\begin{equation}\label{3.8}
(T_k-\widetilde{\lambda}_kI)^2{\bf h}_k=X^{-1}\left[\begin{array}{cc}
(J_1-\widetilde{\lambda}_kI)^2 & \\
 &(J_2-\widetilde{\lambda}_kI)^2\\
\end{array} \right]X{\bf h}_k={\bf 0}.
\end{equation}
Denote $X{\bf h}_k=[{\bf h}^{\rm H},~\widehat{\bf h}^{\rm H}]^{\rm H}$, where ${\bf h}\in\mathbb{C}^{d_k}$ and $\widehat{\bf h}\in\mathbb{C}^{k-d_k}$. As $(J_2-\widetilde{\lambda}_kI)^2$ is nonsingular, it follows from \eqref{3.8} that
$$
\left[\begin{array}{cc}
(J_1-\widetilde{\lambda}_kI)^2 & \\
 &(J_2-\widetilde{\lambda}_kI)^2\\
\end{array}\right]\left[\begin{array}{c}
{\bf h} \\
\widehat{\bf h}\\
\end{array}\right]={\bf 0},
$$
and $\widehat{\bf h}={\bf 0}$.
If $\widetilde{\lambda}_k$ is semi-simple, then
\begin{eqnarray*}
(T_k-\widetilde{\lambda}_kI){\bf h}_k&=&X^{-1}\left[\begin{array}{cc}
J_1-\widetilde{\lambda}_kI& \\
 &J_2-\widetilde{\lambda}_kI\\
\end{array} \right]X{\bf h}_k
=X^{-1}\left[\begin{array}{cc}
O& \\
 &J_2-\widetilde{\lambda}_kI\\
\end{array} \right]X{\bf h}_k\\
&=&X^{-1}\left[\begin{array}{cc}
O& \\
 &J_2-\widetilde{\lambda}_kI\\
\end{array} \right]\left[\begin{array}{c}
{\bf h} \\
{\bf 0}\\
\end{array}\right]={\bf 0},
\end{eqnarray*}
where $O$ stands for a zero matrix. This contradicts to \eqref{3.7}, so we complete the proof.
\end{proof}

However, the inverse is not true, as the following example shows. That is, $\widetilde{\lambda}_k$ is defective does not mean that ${\bf v}\in{\rm span}\{V_k\}$.
\begin{example}
Consider the $4\times 4$ matrix
$$
A=\left[\begin{array}{cccc}
1 & 1 & 1 & 5\\ 0& 1 & 2 & 6\\ 1& 2& 3 & 7 \\ 3 &4 &4 & 8\\
\end{array} \right].
$$
If we choose
$$
V_2=\left[\begin{array}{cc}
1 & 0\\ 0& 1\\ 0& 0\\ 0 &0
\end{array} \right],
$$
then $\widetilde{\lambda}_2=1$ is a Ritz value and ${\bf u}_2=\big[1,0,0,0\big]^{\rm H}$ is a Ritz vector of $A$ in the search subspace ${\rm span}\{V_2\}$, where $\widetilde{\lambda}_2$ is a defective eigenvalue of $V_2^{\rm H}AV_2$. Thus,
${\bf r}_2=A{\bf u}_2-{\bf u}_2=\big[0,0,1,3\big]^{\rm H}~\bot~{\rm span}\{V_2\}$, and
$$
U_{\bot}=\left[\begin{array}{ccc}
0 & 0 & 0\\
1 & 0 & 0\\
0 & 1 & 0\\
0 & 0 & 1\\
\end{array} \right],
\quad
U_{\bot}^{\rm H}(A-\widetilde{\lambda}_2I)U_{\bot}=
\left[\begin{array}{ccc}
0 & 2 & 6\\
2 & 2 & 7\\
4 & 4 & 7\\
\end{array} \right].
$$
Therefore, \eqref{eqn2.1} has a unique solution
$$
U_{\bot}^{\rm H}{\bf v}=\big[-0.571428571428571, -0.428571428571429,0.142857142857143\big]^{\rm H},
$$
and the expansion vector
$$
{\bf v}=\big[0,-0.571428571428571, -0.428571428571429,0.142857142857143\big]^{\rm H}\notin{\rm span}\{V_2\}.
$$
In combination with Theorem \ref{Cor3.1} and Example 2, we conclude that $\widetilde{\lambda}_k$ being defective is not a necessary and sufficient condition for the stagnation of the Jacobi-Davidson method.
\end{example}

Now we focus on (\ref{eqn23}), and see whether the Jacobi-Davidson method based on the alternative correction equation can suffer from stagnation.
\begin{theorem}\label{Thm2.6}
Suppose that the alternative Jacobi-Davidson correction equation {\rm(}\ref{eqn23}{\rm)} has a unique solution, then
\begin{equation}\label{2.11}
{\bf v}\notin {\rm span}\{V_k\}\Longleftrightarrow W\big[W^{\rm H}(A-\widetilde{\lambda}_kI)^{-1}W\big]^{-1}W^{\rm H}{\bf u}_k\notin {\rm span}\{(A-\widetilde{\lambda}_kI)V_k\}.
\end{equation}
\end{theorem}
\begin{proof}
If (\ref{eqn23}) has a unique solution, then $(I-WW^{\rm H})(A-\widetilde{\lambda}_kI){\bf v}=-{\bf r}_k$ and
\begin{equation}\label{2.12}
{\bf v}=-(A-\widetilde{\lambda}_kI)^{-1}{\bf r}_k+(A-\widetilde{\lambda}_kI)^{-1}W\big[W^{\rm H}(A-\widetilde{\lambda}_kI){\bf v}\big].
\end{equation}
Denote ${\bf g}_k=W^{\rm H}(A-\widetilde{\lambda}_kI){\bf v}$. As ${\bf v}\bot{\rm span}\{W\}$, we have from (\ref{2.12}) that
$$
-W^{\rm H}(A-\widetilde{\lambda}_kI)^{-1}{\bf r}_k+W^{\rm H}(A-\widetilde{\lambda}_kI)^{-1}W{\bf g}_k={\bf 0},
$$
and
\begin{equation}\label{2.13}
{\bf g}_k=\big[W^{\rm H}(A-\widetilde{\lambda}_kI)^{-1}W\big]^{-1}W^{\rm H}(A-\widetilde{\lambda}_kI)^{-1}{\bf r}_k.
\end{equation}
Thus, it follows from (\ref{2.12}) and (\ref{2.13}) that
\begin{eqnarray*}
{\bf v}&=&-(A-\widetilde{\lambda}_kI)^{-1}{\bf r}_k+(A-\widetilde{\lambda}_kI)^{-1}W\big[W^{\rm H}(A-\widetilde{\lambda}_kI)^{-1}W\big]^{-1}W^{\rm H}(A-\widetilde{\lambda}_kI)^{-1}{\bf r}_k\\
&=&-{\bf u}_k+(A-\widetilde{\lambda}_kI)^{-1}W\big[W^{\rm H}(A-\widetilde{\lambda}_kI)^{-1}W\big]^{-1}W^{\rm H}(A-\widetilde{\lambda}_kI)^{-1}{\bf r}_k.
\end{eqnarray*}
Let ${\bf f}_k=(A-\widetilde{\lambda}_kI)^{-1}W\big[W^{\rm H}(A-\widetilde{\lambda}_kI)^{-1}W\big]^{-1}W^{\rm H}(A-\widetilde{\lambda}_kI)^{-1}{\bf r}_k$, then
$$
V_kV_k^{\rm H}{\bf v}=-{\bf u}_k+V_kV_k^{\rm H}{\bf f}_k,
$$
and $V_kV_k^{\rm H}{\bf v}={\bf v}$ if and only if $V_kV_k^{\rm H}{\bf f}_k={\bf f}_k$. That is,
$$
(A-\widetilde{\lambda}_kI)^{-1}W\big[W^{\rm H}(A-\widetilde{\lambda}_kI)^{-1}W\big]^{-1}W^{\rm H}(A-\widetilde{\lambda}_kI)^{-1}{\bf r}_k\in{\rm span}\{V_k\},
$$
from which we get (\ref{2.11}).
\end{proof}

\begin{rem}
Theorem \ref{Thm2.6} shows that the Jacobi-Davidson method based on {\rm(}\ref{eqn23}{\rm)} may still suffer from the difficulty of stagnation, provided ${\bf u}_k\in\mathcal{W}\subset{\rm span}\{V_k\}$. On one hand, if $W={\bf u}_k$, then Theorem \ref{Thm2.6} reduces to Theorem \ref{Thm2.5}.
On the other hand, if $W=V_k$,
we have $V_kV_k^{\rm H}{\bf v}=-{\bf u}_k+V_kV_k^{\rm H}{\bf u}_k={\bf 0}$, and {\rm(}\ref{eqn24}{\rm)} can circumvent the difficulty of stagnation \cite{GS}.
\end{rem}

Next, we take into account stagnation of the two-sided Jacobi-Davidson method based on the alternative correction equations (\ref{1.4})--(\ref{1.7}). Using the same trick as Theorem \ref{Thm2.5}, we can prove the following result. 
\begin{theorem}
Under the notation of Section I, for the two-sided and alternating Jacobi-Davidson correction equations, we have that\\
{\rm(i)}~ Suppose that the two-sided Jacobi-Davidson correction equations {\rm(}\ref{1.4}{\rm)} and {\rm(}\ref{1.5}{\rm)} have unique solutions, then
\begin{equation}
{\bf s}~\notin~\mathcal{Q}\Longleftrightarrow (A-\theta I)^{-1}{\bf q}\neq QP^{\rm H}(A-\theta I)^{-1}{\bf q}
\end{equation}
and
\begin{equation}
{\bf t}~\notin~\mathcal{P}\Longleftrightarrow (A^{\rm H}-\bar{\theta} I)^{-1}{\bf p}\neq PQ^{\rm H}(A^{\rm H}-\bar{\theta} I)^{-1}{\bf p},
\end{equation}
where $QP^{\rm H}$ is the oblique projector onto $\mathcal{Q}$ and orthogonal to $\mathcal{P}$, and $PQ^{\rm H}$ is the oblique projector onto $\mathcal{P}$ and orthogonal to $\mathcal{Q}$.\\
{\rm(ii)}~ Suppose that the two-sided Jacobi-Davidson correction equations {\rm(}\ref{1.6}{\rm)} and {\rm(}\ref{1.7}{\rm)} have unique solutions, then
\begin{equation}
{\bf s}~\notin~\mathcal{Q}\Longleftrightarrow (A-\theta I)^{-1}{\bf q}\neq QQ^{\rm H}(A-\theta I)^{-1}{\bf q}
\end{equation}
and
\begin{equation}
{\bf t}~\notin~\mathcal{P}\Longleftrightarrow (A^{\rm H}-\bar{\theta} I)^{-1}{\bf p}\neq PP^{\rm H}(A^{\rm H}-\bar{\theta} I)^{-1}{\bf p},
\end{equation}
where $QQ^{\rm H}$ and $PP^{\rm H}$ are the orthogonal projectors onto $\mathcal{Q}$ and $\mathcal{P}$, respectively.
\end{theorem}

Finally, we briefly discuss the issue of preconditioning for the Jacobi-Davidson method. In practical computations, equation (\ref{1.1}) can be solved approximately by selecting some more easily invertible
approximation for the operator $(I-{\bf u}_k{\bf u}_k^{\rm H})(A-\widetilde{\lambda}_k I)(I-{\bf u}_k{\bf u}_k^{\rm H})$,
or by some preconditioned iterative method in practical calculations \cite{Bai,JD2,JD,Stewart,Van}. More precisely, suppose that we have a preconditioner $K$
for the matrix $A-\widetilde{\lambda}_k I$, such that $K^{-1}(A-\widetilde{\lambda}_k I)\approx I$. Instead of (\ref{1.1}), we have to work efficiently with
\begin{equation}\label{eqn244}
(I-{\bf u}_k{\bf u}_k^{\rm H})K(I-{\bf u}_k{\bf u}_k^{\rm H}){\bf z}=(I-{\bf u}_k{\bf u}_k^{\rm H}){\bf y},
\end{equation}
where ${\bf y}=(A-\widetilde{\lambda}_k I){\bf w}$ and ${\bf w}$ is supplied by the Krylov solver. We point out that Corollary \ref{Thm1} and Theorems \ref{Thm2}--\ref{Cor3.1} also apply to (\ref{eqn244}). In other words, (\ref{eqn244}) may be inconsistent, and the Jacobi-Davidson method with preconditioning may also suffer from stagnation.

\section{Conclusion}
\setcounter{equation}{0}

In this paper, we investigate the correction equation and alternative correction equations of the Jacobi-Davidson type methods for large eigenproblems. We show that
the Jacobi-Davidson correction equation (\ref{1.1}) either has a unique solution or has no solution, and provide a computable necessary and sufficient condition for the solution being unique.
We point out that (\ref{1.1}) may suffer from the difficulty of stagnation, and prove that if the Jacobi-Davidson method stagnates, then the Ritz value is a defective eigenvalue of the projection matrix. A necessary and sufficient condition for expanding the search subspace successfully is presented.

For the alternative correction equation (\ref{eqn23}), we show that it may have a unique solution, infinite number of solutions, and no solution. Moreover, (\ref{eqn23}) can suffer from stagnation provided ${\rm span}\{W\}\neq {\rm span}\{V_k\}$. We conclude that the two-sided Jacobi-Davidson correction equations (\ref{1.4})--(\ref{1.7}) may have a unique solution or have no solution, and the two-sided Jacobi-Davidson method may also stagnates in practice. Therefore, how to seek new correction equations that can cure the above drawbacks is an interesting topic, and deserves further investigation.




\section*{Acknowledgments}
We thank Prof. Zhong-zhi Bai for help discussions.

\end{document}